\documentclass[10pt]{article}
\usepackage[latin2]{inputenc}
\usepackage[T1]{fontenc}
\usepackage{amsthm}
\usepackage[leqno]{amsmath}
\usepackage{amssymb}
\usepackage{amsfonts}
\usepackage{fancyheadings}
\usepackage{color}
\renewcommand{\leq}{\leqslant}
\renewcommand{\geq}{\geqslant}

\newcommand{\rr}{\mathbf{\mathbb{R}}}

\newcommand{\dist}{\operatorname{dist}}

\newtheorem{twr}{Theorem}

\newtheorem{defi}{Definition}

\begin{document}

\title{Some estimations of the \L ojasiewicz exponent for polynomial mappings on semialgebraic sets}
\author{Kacper Grzelakowski}
\maketitle

\begin{abstract} 
We strengthen some estimations of the local and global \L ojasiewicz exponent for polynomial mappings on closed semialgebraic sets obtained by K.Kurdyka, S.Spodzieja and A.Szlachci\'nska in \cite{KurSpo1}.
\end{abstract}

\section*{Introduction}
\L ojasiewicz inequalities are important tools in many different areas of mathematics such as singularity theory, differential analysis or dynamical systems (for example \cite{Brz}, \cite{KurSpo3}, \cite{Nied1}). They first appeared in works of H\"ormander in 1958  \cite{Hor1} and independently in those of \L ojasiewicz in 1958  \cite{Loj1} and 1959 \cite{Loj2}.They were used to prove Schwartz hypothesis that a division of a distibution by a polynomial  \cite{Hor1} and by real analytic function \cite{Loj1} \cite{Loj2} is always possible. Estimates of the {\L}ojasiewicz exponent are nowadays widely used in real and complex algebraic geometry. K.~Kudryka, S.~Spodzieja and A.~Szlachci\'nska in \cite{KurSpo2} have given an estimate of the {\L}ojasiewicz exponent at a point for a continuous semialgebraic mapping on a closed semialgebraic set and an estimate of the {\L}ojasiewicz exponent at infinity for a polynomial mapping on a semialgebraic set. In this paper we show that in case of a polynomial mapping, be it at a point or at infinity, it is possible to obtain slightly stronger results than they have.
\

\emph{\bf\it Keywords:} {\L}ojasiewicz exponent, semialgebraic set, semialgebraic mapping, polynomial mapping.
\section*{\L ojasiewicz Exponent at a point}

Let $X\subset \rr^N$ be a closed semialgebraic set and let $F:X\to\rr^m$ be a  polynomial mapping, such that $0\in X$ and $F(0)=0$.
 Then, there exist positive constants $C, \eta, \varepsilon$ such that the following \L ojasiewicz inequality holds (see \cite{Loj1}):
\begin{equation}\label{e1.1}
|F(x)|\ge C\dist(x,F^{-1}(0)\cap X)^\eta  \quad \text{for} \quad    x\in X, |x|<\varepsilon
\end{equation} 
where $|\cdot|$ is the Euclidean norm and $\dist(x,A)$ is the distance of a point $x$ to the set $A$, i.e. the lower bound of $|x-a|$ for $a\in A$. By convention $\dist (x,\emptyset)=1$.
\begin{defi}
The infimum of the exponents $\eta$ in \eqref{e1.1} is called the \L ojasiewicz exponent of F on the set X at 0 and is denoted by $\mathcal{L}_0(F|X)$.
\end{defi}
Each closed semialgebraic set $X\subset \rr^N$ has a decomposition
\begin{equation*}
 X=X_1\cup\dots\cup X_k 
\end{equation*} 
into the union of closed basic semialgebraic sets
\begin{equation*}
 X_i=\{x\in \rr^N:g_{i,1}(x)\ge0,\dots,g_{i,r_i}(x)\ge0, h_{i,1}(x)=\dots=h_{i,l_i}(x)=0\},
\end{equation*} 
$i = 1,\dots, k $ where $g_{i,1},\dots g_{i,r_i},h_{i,1},\dots,h_{i,l_i} \in\rr[x_1,\dots,x_N]$ (see \cite{BCR}). Assume that $r_i$ is the smallest possible number of inequalities $ g_{i,j}(x)>0$ in the definition of $X_i$ for $i=1,\dots,k$.  Denote by $r(X)$ the minimum of $\max\{r_1,\dots,r_k\}$ over all decompositions into unions of sets of X.  Obviously r(X) = 0 means that $X$ is an algebraic set.  Denote by $\kappa(X)$ the minimum of the numbers
 \begin{equation*}
\max\{\deg g_{1,1},\dots,\deg g_{k,r_k},\deg h_{1,1},\dots,\deg h_{k,l_k}\}
\end{equation*} 
over all decompositions of X into the union of sets, provided $r_i \le r(X)$. By $\deg F$ we mean the maximum of the degrees of the components of the mapping F. 

First aim of this paper is to prove the following theorem:
\begin{twr} 
Let $X\subset \rr^N$ be a closed semialgebraic set such that $0\in X$ and let $F:X\to\rr^m$ be a nonzero polynomial mapping such that $F(0)=0$. Set $r=r(X)$ and $d=\max\{\kappa(X),\deg F\}$. Then:
\begin{equation}\label{eqmain1}
\mathcal{L}_0(F|X)  \le d(6d-3)^{N+r+m-1}.
\end{equation}
\end{twr}

In \cite[Corollary 2.2]{KurSpo2} Kurdyka, Spodzieja and Szlachci\'nska proved that:
$$
\mathcal{L}_0(F|X)  \le d(6d-3)^{N+R+m-1}
$$
with $R=r(X)+r(graphF)$. Actually in \cite{KurSpo2} there is no $m$ in the inequality but this should be considered a typographical error. Thus in our theorem we do improve their estimation by using $r=r(X) $ instead of  $R=r(X)+r(graphF)$. For this paper to be self-contained and more clear we will have to repeat some of the argumentation from \cite{KurSpo2} for polynomial mappings on semialgebraic sets. 

 In the proof of the Theorem 1 we will use the result obtained in \cite[Corollary 8]{KurSpo1} regarding \L ojasiewicz exponent in the case of two algebraic sets. Let  $X$ and $Y$ be algebraic subsets of $\rr^M$ described by polynomials of degree not greater than $d$. Let $a\in\rr^M$. Then there exists a positive constant $C$ such that:
 \begin{equation}\tag{KS1}\label{KS1}
\dist(x,X)+\dist(x,Y)\ge C\dist(x,X\cap Y)^{d(6d-3)^{M-1}}
\end{equation}
in a neighbourhood $U\subset \rr^M$ of $a$. We will also use another result from \cite[Corollary 6]{KurSpo1}. For a real polynomial mapping $F:\rr^N \to \rr^m$ such that $d=\deg F$ we have
 \begin{equation}\tag{KS2}\label{KS2}
\mathcal{L}_0(F)  \le d(6d-3)^{M-1}.
\end{equation}

\begin{proof}[Proof of Theorem 1]If $d=1$ then the statement is obvious. Let us assume that $d\geq 2$. 
It suffices to consider the case when $X$ is a basic semialgebraic set. The set X was originally in $\rr^N$ but since we will operate in $\rr^N\times\rr^m$ we need the set $X\times\{0\}\subset\rr^N\times\rr^m.$ Not to overuse the notation from now on we will use $X$ to denote $X\times\{0\}\subset\rr^N\times\rr^m.$ So let:
\begin{multline}\label{defX}
X:=\{(x,z)\in \rr^N\times \rr^{m}:g_{1}(x)\ge0,\dots,g_{r(X)}(x)\ge0, \\ h_{1}(x)=\dots=h_{l}(x)=0, z=0\},
\end{multline}
\begin{multline*}
Y:=\{(x,z)\in \rr^N\times \rr^{m}:g_{1}(x)\ge0,\dots,g_{r(X)}(x)\ge0, \\ h_{1}(x)=\dots=h_{l}(x)=0, z=F(x)\}.
\end{multline*} 
Now, let us define a mapping $G:\rr^{N}\times \rr^{r}\to\rr^{r}$ by:
\begin{equation*} 
G(x,y_1,\dots, y_{r}):=\{g_{1}(x)-y^2_1, \dots,  g_{r}(x)-y_{r}^2\}, 
\end{equation*}
and then sets:
\begin{multline*}
A:=\{(x,z,y_1,\dots, y_{r})\in \rr^{N}\times\rr^{m}\times\rr^{r}: \\ G(x,y)=0, h_{1}(x)=\dots=h_{l}(x)=0, z=0\},
\end{multline*}
 \begin{multline*}
B:=\{(x,z,y_1,\dots, y_{r})\in \rr^{N}\times\rr^{m}\times\rr^{r}: \\ G(x,y)=0, h_{1}(x)=\dots=h_{l}(x)=0, z=F(X)\}.
\end{multline*}
Then A and B are algebraic sets and $\pi(A)=X, \pi(B)=Y$, where 
\begin{equation*}
\pi : \rr^{N+m}\times\rr^r \to  \rr^{N+m}, \quad \pi(x,z,y_1,\dots,y_r)=(x,z).
\end{equation*}
From the definitions of A and B we obtain:
\begin{equation}\label{p1.1}
\forall_{(x,0)\in X}\;\exists_{z\in\rr^m}\; \exists_{y\in\rr^r} (x,0,y)\in A \wedge (x,z,y)\in B.
\end{equation}
Since A and B are algebraic sets defined by polynomials of degree not greater than $d$ then by (\ref{KS1}), for sets $A, B$ there exists a positive constant $C$ such that:
 \begin{equation}\label{p1.4}
\dist((x,0,y),A)+\dist((x,0,y),B)\ge C\dist((x,0,y),A\cap B)^{d(6d-3)^{N+r+m-1}}
\end{equation}
in some neighbourhood $W$ of $0\in\rr^{N+m+r}$. For any $(x,z,y)\in\rr^{N+m+r}$ we have:
\begin{equation}\label{p1.4.1}
\dist((x,z,y),A\cap B)\ge \dist((x,z),X\cap Y) .
\end{equation}

We can now assume that $g_{i,j}(0)=0$ for any $i,j$. Indeed, if $g_{i,j}(0)<0$ for some $i,j$ then $0\notin X \cap Y$ which contradicts the assumption.If $g_{i,j}(0)>0$ for some $i,j$ then it is safe to omit this inequality in the definition of $X$ (and $Y$) and the germ of 0 at $X$ or $Y$ will not change. If $g_{i,j}(0)>0$ for any $i,j$, then we can reduce our assertion to (\ref{KS2}). So, there exists a neighbourhood $V=V_1\times V_2 \subset W$ of $0\in \rr^{N+m+r}$ where $V_{1}\subset \rr^{N+m}$ and $V_2\subset \rr^r$ such that:
\begin{equation}\label{p1.5}
\forall_{(x,0,y)\in A:\; (x,0)\in \rr^{N+m},\; y\in\rr^r} \quad (x,0)\in X\cap V_{1}\Rightarrow y\in V_2,
\end{equation}
and
\begin{equation}\label{p1.6}
\forall_{(x,z,y)\in B:\; (x,z)\in \rr^{N+m},\; y\in\rr^r} \quad (x,z)\in Y\cap V_{1}\Rightarrow y\in V_2.
\end{equation}

Note, that since $A$ and $B$ were defined by $N+r$ identical coordinates $x,y$ and differ only in $m$ of them $z$.
This explains why in (\ref{p1.5}) and in (\ref{p1.6}) we were able to consider the same neighbourhood $V_2\subset\rr^r$. 

Since F is a continuous mapping there exist neighbourhoods $U_1 \subset \rr^N$ and $U_2 \subset \rr^m$ of the origin such that $U_1 \times U_2 \subset V_1 $ and for every $x \in U_1$ we have  $z=F(x) \in U_2$. Then $(U_1 \times U_2) \times V_2 \subset W$. Consider some $x\in U_1$. By \eqref{p1.1} there exist $z\in \rr^m$ and $y\in \rr^r$ such that $(x,0,y)\in A$ and $(x,z,y)\in B$. Then, by (\ref{p1.5}) and (\ref{p1.6}) we see that $(x,0,y)\in V$.  
Let us observe that:
\begin{equation*}
|F(x)|=|(x,0)-(x,z)|=|(x,0,y)-(x,z,y)| \geq \dist ((x,0,y),B).
\end{equation*}
Since $(x,z,y)\in B$, and $(x,0,y)\in A$ then, from the above: 
\begin{equation*}
|F(x)|\ge \dist((x,0,y),A) + \dist((x,0,y), B).
\end{equation*}
 Since $A,B\in \rr^{N+m+r}$, by (\ref{p1.4}) and by (\ref{p1.4.1}), we obtain :
\begin{multline*}
|F(x)| \ge \dist((x,0,y),A)+ \dist((x,0,y), B) \\  \ge C\dist((x,0,y),A\cap B)^{d(6d-3)^{N+r+m-1}} \\ \ge C\dist((x,0),X\cap Y)^{d(6d-3)^{N+r+m-1}}.
\end{multline*}
Since $X\cap Y = (F^{-1}(0)\times{0})$  we obtain the assertion.
\end{proof}

\section*{The \L ojasiewicz Exponent at infinity}
The second result of this paper concerns the global {\L}ojasiewicz inequality and the \L ojasiewicz exponent  of a polynomial mapping at infinity.
\begin{defi}
Assume that a closed semialgebraic set $X\subset \rr^N$ is unbounded. By the \L ojasiewicz exponent at infinity of a polynomial mapping $F:X \to \rr^m$ we mean the supremum of the exponents $\eta$ in the following inequality:
$$
|F(x)|\geq C|x|^{\eta} \quad for \quad x\in X, |x|\geq c 
$$
for some positive constants $C, c$. We denote it by $\mathcal{L}_{\infty}(F|X).$ In case $X=\rr^N$ we call this exponent the \L ojasiewicz exponent at infinity and denote it by $\mathcal{L}_{\infty}(F).$
\end{defi}
In \cite[Corollary 10]{KurSpo1} it is proved, that for a polynomial mapping $F=(f_1,\dots,f_m):\rr^N\to\rr^m$ of degree $d$ of a real algebraic set $X$ we have:
\begin{equation}\tag{KS3}\label{KS3}
|F(x)| \geq C\Bigg(\frac{\dist(x,F^{-1}(0)\cap X)}{1+|x|^2}\Bigg)^{d(6d-3)^{M-1}}\quad \text{for} \  x\in \rr^M.
\end{equation}
Using this, in \cite[Corollary 3.4]{KurSpo2} it has been shown that for a polynomial mapping $F$ on a closed semialgebraic set $X$ the following inequality holds:
\begin{equation*}
|F(x)| \geq C\Bigg(\frac{\dist(x,F^{-1}(0)\cap X)}{1+|x|^2}\Bigg)^{d(6d-3)^{N+R-1}} \quad \text{for} \  x\in \rr^N,
\end{equation*}
where $R=2r(X)$. We are again going to show that this estimate can be improved by substituting $R$ with $r=r(X)$ and also by adding $m$.
\begin{twr}\label{t1.2} 
Let $F:X\rightarrow \rr^m$ be a polynomial mapping, where $X\subset\rr^N$ is a closed semialgebraic set. If $D=\max\{2,\kappa(X)\}, d=\max\{\deg F,D\}$ and $ r=r(X)$ then:
\begin{equation} \label{p2.1}
|F(x)| \geq C\Bigg(\frac{\dist(x,F^{-1}(0)\cap X)}{1+|x|^D}\Bigg)^{d(6d-3)^{N+r+m-1}} \quad  \text{\rm for} \  x\in \rr^M.
\end{equation}
If additionally $X$ is unbounded set and $ F^{-1}(0)\cap X$ is a compact set, then:
\begin{equation} \label{p2.2}
\mathcal{L}_{\infty}(F|X)\geq - \frac{D}{2} d(6d-3)^{N+r+m-1}.
\end{equation}
\end{twr}
\begin{proof}[Proof of Theorem 2]
Again we will repeat the argumentation from \cite{KurSpo2}.
Also, as in the previous proof we will consider the set $X\times\{0\}\subset\rr^N\times\rr^m$ defined by (\ref{defX}), and denote it simply by $X$ to avoid overuse of notation.
Let $H:\rr^{N+r+m}\to \rr^{r+m+l}$ be a polynomial mapping defined by:
\begin{equation*}
H(x,z,y)=(F(x,z),G(x,y),h_{1,1}(x),\dots,h_{1,l}(x)), \quad    x\in\rr^N, z\in\rr^m , y\in\rr^r
\end{equation*}
with $G$ being defined as in the previous proof. Then  $\deg H\leq d$. Let $V=F^{-1}(0) \cap X$ and $Z=H^{-1}(0)$. Obviously $Z$ is an algebraic set.  By (\ref{KS3}) for some positive constant $C$ we have:
\begin{equation*}
|H(x,z,y)|\geq C\Bigg(\frac{\dist((x,z,y),Z)}{1+|(x,0,y)|^2}\Bigg)^{d(6d-3)^{N+r+m-1}}  
\end{equation*}
for $ (x,0,y)\in\rr^N\times\rr^m\times\rr^r$. Obviously $\dist ((x,z,y),Z) \geq \dist((x,z),V)$ and thus:
\begin{equation} \label{p2.3}
|H(x,z,y)|\geq C\Bigg(\frac{\dist((x,z),V)}{1+|(x,0,y)|^2}\Bigg)^{d(6d-3)^{N+r+m-1}}
\end{equation}
 for $(x,0,y)\in\rr^N\times\rr^m\times\rr^r$. 
It is easy to observe that there exist constants $C_1 \geq 0,  R_1 \geq 1$ such that for $(x,0,y)\in A$ with $|(x,0,y)|\geq R_1$ we have $C_1|y|^2\leq|(x,0)|^D$. Since $D\geq 2$, for a constant $C_2 > 0$ we have $|(x,0,y)|\leq C_2|(x,0)|^{D/2}$ for $(x,0,y)\in A, |(x,0,y)|\geq R_1$. Hence, from (\ref{p2.3}) we obtain (\ref{p2.1}) for $(x,0)\in X$, $|(x,0)|>R_1$. Again, by diminishing $C$ if necessary we obtain (\ref{p2.1}) for all $(x,0)\in X$.

Now, let us prove the second assertion of Theorem 2. To do this, we will need yet another result from \cite{KurSpo1}, namely [Corollary 11]. The authors have shown that if $F=(f_1,\dots,f_m):\rr^N \to \rr^m$ is a polynomial mapping of degree $d\geq 1$, and $F^{-1}(0)$ is a compact set then:
\begin{equation}\tag{KS3} \label{KS3}
\mathcal{L}_{\infty}(F)  \geq -d(6d-3)^{n-1}.
\end{equation} 
Since $X$ is unbounded  we may assume that so is $A$. Since $V$ is compact, so is $H^{-1}(0)$. By (\ref{KS3}) we have $\mathcal{L}_{\infty}(H)  \geq -d(6d-3)^{N+r+m-1}$, in particular, for some constants $C,R > 0$,
\begin{equation*}
|H(x,0,y)|\geq C|(x,0,y)^{-d(6d-3)^{N+r+m-1}}\quad  \textrm{for }   (x,0,y)\in A, |(x,0,y)|\geq R   
\end{equation*}
Since $|(x,0,y)|\leq C_2|(x,0)|^{D/2}$ for $(x,0,y)\in A$, $|(x,0,y)|\geq R_1$, then for some constant $C_3>0$:
\begin{equation*}
|F(x,0)|=|H(x,0,y)|\geq C_3|x|^{-\frac{D}{2}d(6d-3)^{N+r+m-1}}\quad  \textrm{for } (x,0,y)\in A, |(x,0,y)|\geq R
\end{equation*}
and also $\mathcal{L}^{\rr}_{\infty}(F|X)  \geq -\frac{D}{2}d(6d-3)^{N+r+m-1}$, which ends the proof.
\end{proof}
\section*{Acknowledgement}
The author wishes to thank Stanis\l aw Spodzieja who has provided both inspiration and support in the process of writing this paper and Tadeusz Krasi\'nski whose insight has been more than helpful.


\begin{thebibliography}{999} 
\bibitem{BCR} J. Bochnak, M. Coste, M-F. Roy, \emph{Real algebraic geometry},  Springer-Verlag, Berlin, 1998. 
\bibitem{Brz}  S. Brzostowski, \emph{The \L ojasiewicz exponent of semiquasihomogeneous singularities}. Bull. Lond. Math. Soc. 47 (2015), no. 5, 848-852.
\bibitem{Hor1} L. H\"ormander, \emph{On the division of distributions by polynomials.} Ark. Mat. 3, 555-568 (1958)
\bibitem{KurSpo1} K. Kurdyka, S. Spodzieja, {\it Separation of real algebraic sets and the \L ojasiewicz exponent}. Proc. Am. Math. Soc 142(9), 3089-3102(2014)
\bibitem{KurSpo2} K. Kurdyka, S. Spodzieja, A. Szlachci\'nska,  {\it Metric Properties of Semialgebraic Mappings}. Discrete Comput Geom (2013)
\bibitem{KurSpo3}  K. Kurdyka, S. Spodzieja, \emph{Convexifying positive polynomials and sums of squares approximation}. SIAM J. Optim. 25 (2015), no. 4, 2512-2536.
\bibitem{Loj1} S. \L ojasiewicz, {\it Division d'une distibution par une fonction analytique de variables r\'eelles}. C.R. Acad. Sci. Paris 246, 683-686 (1958)
\bibitem{Loj2} S. \L ojasiewicz,  {\it Sur le probl\`eme de la division}. Studia Math. 18, 81-136 (1959)
\bibitem{Nied1} L. Niedermann, {\it Hamiltionian stability and subanalytic geometry.} Ann. Inst. Fourier (Grenoble) 56(3), 795-813 (2006)
\end{thebibliography}
\end{document}